\documentclass{article}
\usepackage[utf8]{inputenc}
\usepackage{amsmath}
\usepackage{amsfonts}
\usepackage{amsthm}
\usepackage{amssymb}
\usepackage{mathtools}
\usepackage{dsfont}
\usepackage{authblk}
\DeclarePairedDelimiter{\ceil}{\lceil}{\rceil}
\newtheorem{theorem}{Theorem}[section]

\newtheorem{lemma}[theorem]{Lemma}

\newtheorem{definition}[theorem]{Definition}

\newcommand{\hsp}{\hspace{.5mm}}
\newcommand{\be}{\begin{equation}}
\newcommand{\ee}{\end{equation}}
\newcommand{\dmt}{\hsp \text{d}\mu(t)}
\numberwithin{equation}{section}
\title{On the boundedness of certain generalized Hilbert operators in $\ell^p$}
	\author{Nikolaos Athanasiou\thanks{American College of Thessaloniki, V. Sevenidi 17, Pylaia 55535, Thessaloniki, Greece. \hspace{13mm} E-mail address: nikathan@act.edu}}

\date{May 2022}

\begin{document}

\maketitle
\begin{abstract}
   \noindent The Hilbert matrix $\mathcal{H}_{n,m} = (n+m+1)^{-1}$ has been extensively studied in previous literature. In this paper we look at generalized Hilbert operators arising from measures on the interval $[0,1]$, such that the Hilbert matrix is obtained by the Lebesgue measure. We provide a necessary and sufficient condition for these operators to be bounded in $\ell^p$ and calculate their norm.
\end{abstract}

\section{Introduction}
\subsection{Motivation and history}
The study of the classical Hilbert operator has its origins on the classical inequality of the same name. 

\begin{theorem}{\textbf{(Hilbert's inequality)}} Suppose that  $(a_m)_{m=1}^{\infty}$ and $(b_n)_{n=1}^{\infty}$ are sequences of non-negative terms and let $p,\hsp q> 1$ be H\"older conjugates, i.e. such that $\frac{1}{p}+\frac{1}{q}=1$. Assume that $(a_m)\in \ell^p$ and $(b_n)\in \ell^q$. Then there holds
\[ \sum_{m=1}^{\infty}\hsp \sum_{n=1}^{\infty}\frac{a_m\hsp b_n}{m+n} \leq \frac{\pi}{\text{sin}\left(\frac{\pi}{p}\right)} \left(\sum_{m=1}^{\infty}a_m^p\right)^{\frac{1}{p}}\left(\sum_{n=1}^{\infty}b_n^q\right)^{\frac{1}{q}}.\]
\end{theorem}\par\noindent It was Hilbert who first proved the above for the special case $p=q=2$ in one of his lectures on integral equations. His proof was published by Weyl (1908). The first proof for general $p,\hsp q>1$ was given by Hardy and Riesz \cite{HarRie}. Perhaps the best reference for the proof is \cite{Hardy 1}. A sharp version of Hilbert's inequality is the following: For $(a_m)_{m=0}^{\infty}$ and $(b_n)_{n=0}^{\infty}$ both sequences of non-negative numbers such that $(a_m)\in \ell^p$ and $(b_n) \in \ell^q$, there holds 
\be \label{11} \sum_{m=0}^{\infty}\hsp \sum_{n=0}^{\infty}\frac{a_m\hsp b_n}{m+n+1} \leq \frac{\pi}{\text{sin}\left(\frac{\pi}{p}\right)} \left(\sum_{m=0}^{\infty}a_m^p\right)^{\frac{1}{p}}\left(\sum_{n=0}^{\infty}b_n^q\right)^{\frac{1}{q}}.\ee The Hilbert operator and its associated Hilbert matrix are now introduced as follows: Consider the operator $H$ on $\ell^p$ given by
\[ H: (a_n)_{n=0}^{\infty} \mapsto \left( \sum_{m=0}^{\infty} \frac{a_m}{m+n+1}\right)_{n=0}^{\infty}. \] The fact that this is a bounded operator on $\ell^p$ for $p>1$ follows from Hilbert's inequality as follows. 

\vspace{3mm}

\par \noindent We recall that the norm of $(a_n) \in \ell^p$ is $\lVert (a_n) \rVert_{\ell^p} = \left( \sum_{n=0}^{\infty} a_n^p \right)^{\frac{1}{p}}$.Moreover there must hold \[ \lVert (a_n) \rVert_{\ell^p} = \sup_{h \in (\ell^p)^* , \hsp \hsp \lVert h \rVert \leq 1} h((a_n)), \]where $h$ is any element of the dual space $\ell^q$ with norm $ 1$. Therefore \[ \lVert (a_n) \rVert_{\ell^p} = \sup_{(b_n) \in \ell^q, \hsp \hsp \rVert (b_n) \rVert_{\ell^q} \leq 1 } \bigg\lvert \sum_{n=0}^{\infty} a_n b_n\bigg\rvert.\]By fixing an arbitrary number $k$ and taking the supremum over all sequences $(b_n) \in \ell^q$ such that $\lVert (b_n) \rVert_{\ell^q} =k$, we arrive at the desired result.
\vspace{3mm} The associated matrix 

\[ H = \begin{pmatrix} 1 & \frac{1}{2} & \frac{1}{3} & \dots \\ 
\frac{1}{2} & \frac{1}{3} & \frac{1}{4} & \dots \\
\frac{1}{3} & \frac{1}{4} & \frac{1}{5} & \dots \\
\vdots & \vdots & \vdots & \ddots
\end{pmatrix}
\] acting on column vectors whose entries are terms of a sequence in $\ell^p$ is another way to look at the operator (we shall not be making a notational distinction between $H$ and its associated matrix). Further historical notes and properties of the Hilbert matrix can be found in \cite{Daskalo}.

\vspace{3mm}
\par \noindent Notice that each entry in $H$ can be written as \[ H_{n,\hsp k} =  \int_0^1 g_{n,k}(t) \hsp \text{d}t, \]where $g_{n,k}(t) := \binom{n+k}{k} \hsp (1-t)^n \hsp t^k$.
Motivated by this integral representation, in the current paper we shall be looking at a class of generalized Hilbert operators, arising from  replacing $\text{d}t$ with finite, positive Borel measures. The definition is as follows.

\begin{definition}
Let $\mu$ be a finite, positive Borel measure on $[0,1]$. We define the generalized Hilbert matrix $\mathcal{C}^{\mu}$ by \begin{equation}
    \mathcal{C}^{\mu}_{n, k} := \binom{n+k}{k} \hsp \int_0^1 (1-t)^n \hsp t^k \hsp \text{d}\mu(t),
\end{equation} for $n,k$ integers ranging from 0 to $\infty$.

\end{definition}\noindent The generalized Hilbert matrix can be viewed as an operator on $\ell^p$ given by $\ell^p \ni (c_n)_{n=0}^{\infty} \mapsto (d_n)_{n=0}^{\infty}$, with \begin{equation}
    d_n := \sum_{k=0}^{\infty}\mathcal{C}^{\mu}_{n,k}c_k.\end{equation}
\vspace{3mm}

\par\noindent Let us, finally, note that the corresponding question to the one addressed in this paper for the case of generalized C\'esaro operators was tackled in \cite{Hardy} and throughout this work, our general approach significantly borrows from the ideas presented in it.

\subsection{Preliminaries}
In this subsection we will be introducing a series of lemmata which will become relevant later on. 
\begin{lemma}
There holds \[ \sum_{n=0}^{\infty} \binom{n+m}{m} \hsp s^n = (1-s)^{-m-1},\]for all $s \in (0,1).$\label{lem13}
\end{lemma}
\begin{proof}
That the left-hand side converges (absolutely) is a direct consequence of the ratio test for series. One can use induction on $m$. For $m=0$, the result holds. Assume the result holds for $m= k$. Then we have \[ (1-s)^{-m-2} = \frac{1}{1-s} (1-s)^{-m-1} = \left(1+s+s^2+\dots\right) \hsp \sum_{n=0}^{\infty} \binom{n+m}{m} s^n. \] Since both series above converge absolutely, their product is given by the series 

\[ \sum_{n=0}^{\infty} \left(\sum_{l=0}^{n} \binom{m+l}{m} \right)\hsp s^n, \]which is convergent by Mertens' theorem on the Cauchy product of series. Therefore, all one has to show is that \[ \sum_{l=0}^{n} \binom{m+l}{m}  = \binom{m+n+1}{m+1}, \]which itself can be shown by a straightforward induction, using the fact that \[ \binom{m+n+1}{m+1} +\binom{m+n+1}{m} = \binom{m+n+2}{m+1}.\]
\end{proof}

\begin{lemma}\label{lem14}
There holds \[ \frac{1-t}{1-t\hsp e^{-x}} \geq e^{-\frac{t}{1-t}\hsp x},\] for $x$ non-negative and $t\in [0,1)$.
\end{lemma}
\begin{proof}
One can consider, for a fixed $t$, the function \[f_t(x) = (1-t\hsp e^{-x}) e^{-\frac{t}{1-t}\hsp x}.\]Differentiating with respect to $x$ we see that the only critical point is at $x=0$, whence the result follows.
\end{proof}

\begin{lemma}
\label{lem15} There holds \[ \frac{1}{\Gamma(\omega)} \int_0^{\infty} \text{e}^{-a\hsp x} x^{w-1}\hsp \text{d}x = a^{-w},\]for any $a>0$ and $w>0$.
\end{lemma}
\begin{proof}
As long as $a$ is positive, the integral is well-defined and convergent. Recall that \[ \Gamma(\omega) = \int_0^{\infty} e^{-x} x^{w-1} \hsp\text{d}x. \] Setting $u=a\hsp x$, we have $\text{d}x = \frac{\text{d}u}{a}$ and the integral becomes \begin{equation*}
        \int_0^{\infty} e^{-u} u^{w-1} a^{-w} \hsp \text{d}u = a^{-w}\hsp \Gamma(\omega).
\end{equation*}The result follows.
\end{proof}
\section{Boundedness of $\mathcal{C}^{\mu}$ on $\ell^p$}
In the current section we shall prove two theorems regarding the boundedness of $\mathcal{C}^{\mu}$ as an operator on $\ell^p$ spaces. The first concerns the case where there are no Dirac masses at $0$ or $1$, while the latter allows for the existence of such masses at the endpoints of $[0,1]$.

\begin{theorem}
Fix $1\leq p \leq \infty$. If $\mu$ has no Dirac masses at $0$ or $1$, then the matrix $\mathcal{C}^{\mu}$ is bounded as an operator from $\ell^p$ to itself if and only if\[ \mathcal{N}^{\mu}_{p} := \int_0^1 \frac{1}{(1-t)^{1-\frac{1}{p}}\hsp t^{\frac{1}{p}}} \hsp \text{d}\mu(t) < \infty.\]Moreover, if that is the case, the operator norm equals $\mathcal{N}^{\mu}_p$ (here we understand that when $p=\infty$, the integral becomes $\int_0^1 \frac{1}{1-t}\hsp \text{d}\mu(t)$).
   %\item If $\mu$ contains Dirac masses $c_0$ at $0$ and $c_1$ at $1$, then $\mathcal{C}^{\mu}$ can only be bounded when $c_1\neq 0 = c_0$ and $p=1$, in which case a necessary and sufficient condition for it to be bounded is the finiteness of $\mu_{(0,1]} \left(\frac{1}{t}\right) \text{d}\mu(t)$. Moreover, this expression equals the operator norm, when finite.
\end{theorem}
\begin{proof}
 %For now we postpone the treatment of this case until the end of the proof and focus on the case where there are no Dirac masses at $0$ or $1$, so that $\mu\left((0,1)\right) = \mu\left([0,1]\right)$.
%\vspace{3mm}
%\begin{center}
%   \textbf{Case 1: No Dirac masses at $0$ or $1$}.
%\end{center}
%\vspace{3mm}

\par \noindent Suppose that $\int_0^1 \frac{1}{(1-t)^{\frac{p-1}{p}}\hsp t^{\frac{1}{p}}} \hsp \text{d}\mu(t)$ is finite and there are no Dirac masses at $0$ or $1$. We will, in the first instance, show that $\mathcal{C}^{\mu}$ is bounded from $\ell^p$ to $\ell^p$ for $1\leq p < \infty$ and we will calculate its norm. We will then treat the case $p=\infty$. 

\vspace{3mm}

\par \noindent We begin by defining, for a given sequence $(a_n)_{n=0}^{\infty} \in \ell^p$, the function $e_n: [0,1] \to \mathbb{R}$ given by \begin{equation} \label{ent}
    e_n(t) = \sum_{m=0}^\infty \binom{n+m}{m} t^m \hsp (1-t)^n \hsp a_m.
\end{equation}We claim that $e_n$ is well-defined on $[0,1]$ and that there holds, for $t\in (0,1)$, \begin{equation}
    \sum_{n=0}^{\infty}e_n^p \leq \frac{1}{(1-t)^{p-1} \hsp t} \sum_{n=0}^{\infty} a_n^p.
\end{equation}To check this, we first look at the case $p>1$. Here, H\"older's inequality gives 

\begin{equation}
    e_n(t) \leq \left( \sum_{m=0}^{\infty} \underbrace{{\binom{n+m}{m}}^q \hsp t^{q\hsp m} \hsp (1-t)^{q\hsp n} }_{:= \hsp B_m} \right)^{\frac{1}{q}} \left(\sum_{m=0}^{\infty} a_m^p \right)^{\frac{1}{p}}, \label{step1}
\end{equation}where $q =\frac{p}{p-1}$ is the H\"older conjugate of $p$, i.e. $\frac{1}{q} + \frac{1}{p} =1$.
To check that the first term on the right-hand side in \eqref{step1} is convergent, the ratio test for series gives

\begin{equation}
    \frac{B_{m+1}}{B_m} = \left( \frac{\binom{n+m+1}{m+1} \hsp t^{m+1}
    \hsp (1-t)^n}{\binom{n+m}{m} \hsp t^m \hsp (1-t)^n} \right)^q = \left(\frac{m+n+1}{m+1} \hsp t\right)^q \to t^q,
\end{equation}as $m\to \infty$. Hence, for $t\in [0,1)$, we have \[\lim_{m \to \infty} \frac{B_{m+1}}{B_m} = t^q<t<1\] and therefore the series is (absolutely) convergent. Finally, for $t=1$, the  partial sums $\sum_{m=0}^K B_m =0$ and hence the series equals zero. This shows that for  all $t \in[0,1]$, $e_n(t)$ is well-defined. 

\vspace{3mm}
\par \noindent For $p=1$, it is enough to show that, for a fixed $t \in [0,1]$ and for any fixed $n$, the sequence \[ c_m(n;t) := \binom{n+m}{m}\hsp t^m (1-t)^n \]is in $\ell^{\infty}$. Consider a biased coin with probability of heads $t$ and probability of tails $1-t$. The probability of landing $m$ heads in $(n+m)$ throws for this coin is precisely $c_m(n;t)$. Hence each term is bounded by 1 and the result follows. This proves that $e_n(t)$ is well-defined when $(a_n) \in \ell^1$.
\vspace{3mm}

\noindent By H\"older's inequality again, we have

\begin{equation}
    \begin{split}
        e_n^p(t) &\leq \left(\sum_{m=0}^{\infty} \binom{n+m}{m}\hsp (1-t)^n \hsp t^m \hsp a_m^p \right) \left( \sum_{m=0}^{\infty}\binom{n+m}{m} \hsp (1-t)^n\hsp t^m \right)^{p-1} \\&\leq \frac{1}{(1-t)^{p-1}} \sum_{m=0}^{\infty} \binom{n+m}{m}\hsp (1-t)^n \hsp t^m \hsp a_m^p.
    \end{split}
\end{equation}Summing over $n$, we get (for $0< t<1$),

\begin{equation}
    \begin{split}
        \sum_{n=0}^{\infty} e_n^p &\leq \sum_{n=0}^{\infty}\sum_{m=0}^\infty \frac{1}{(1-t)^{p-1}} \hsp \binom{n+m}{m}(1-t)^n\hsp t^m\hsp a_m^p \\&\leq\sum_{m=0}^{\infty} \frac{1}{(1-t)^{p-1}}\hsp t^m\hsp a_m^p \hsp \sum_{n=0}^{\infty}\binom{n+m}{m}\hsp(1-t)^n \\ &\leq \sum_{m=0}^{\infty} \frac{1}{(1-t)^{p-1}\hsp t}a_m^p, \label{use}
    \end{split}
\end{equation}where we have used Lemma \ref{lem13}. We now have \begin{equation*}
    \begin{split}
        b_n &= \sum_{m=0}^{\infty} \int_{0}^1 \binom{n+m}{m}\hsp t^m \hsp (1-t)^n\hsp a_m \hsp\text{d}\mu(t) \\ &=\int_{0}^1  \sum_{m=0}^{\infty} \binom{n+m}{m}\hsp t^m \hsp (1-t)^n\hsp a_m \hsp\text{d}\mu(t) \hspace{5mm} \text{(by the \footnote{Dominated Convergence Theorem}D.C.T.)} \\ &= \int_0^1 e_n(t)\hsp\text{d}\mu(t).
    \end{split}
\end{equation*}
%\par\noindent Since, for all $m$ and $n$, the function $d_m(n; \hsp t) := c_m(n;\hsp t) \hsp a_m$ is defined for all $t \in [0,1]$ and moreover there are no Dirac masses at $0$ or $1$, we have \[ \int_0^1 d_m(n; \hsp t) \hsp \text{d}\mu(t) = \int_{(0,1)} d_m(n; \hsp t) \hsp \text{d}\mu(t), \] for all $m$ and $n$ (including, crucially, the case $n=0$). Therefore,
By Minkowski's inequality, we have 
\begin{equation}
       \left(\sum_{n=0}^{\infty}b_n^p\right)^{\frac{1}{p}} \leq \int_0^1 \left( \sum_{n=0}^{\infty} e_n^p\right)^{\frac{1}{p}} \hsp \text{d}\mu(t) \leq \int_0^1 \frac{1}{t^{\frac{1}{p}}\hsp (1-t)^{\frac{p-1}{p}}}\hsp \left(\sum_{n=0}^{\infty}a_n^p\right)^{\frac{1}{p}} \hsp\text{d}\mu(t),
\end{equation}or equivalently
\begin{equation}
    \lVert (b_n) \rVert_{\ell^p} \leq \left(\int_0^1 \frac{1}{t^{\frac{1}{p}}\hsp (1-t)^{\frac{p-1}{p}}} \hsp\text{d}\mu(t)\right)\lVert (a_n) \rVert_{\ell^p}.
\end{equation}To show that $\mathcal{N}_p^{\mu}$ is the optimal constant (and hence the norm of the corresponding operator, when it is finite) we work as follows. Fix any $\eta>0$ small. There exists a $\delta >0$ such that \begin{equation}
    \int_{\delta}^{1-\delta}  \frac{1}{t^{\frac{1}{p}}\hsp (1-t)^{\frac{p-1}{p}}} \hsp\text{d}\mu(t) > (1-\eta)\hsp \mathcal{N}_{p}^{\mu}.
\end{equation}That such a $\delta$ exists follows from the monotone convergence theorem, when applied for example to the functions $f_n(t) = \frac{1}{t^{\frac{1}{p}}\hsp (1-t)^{\frac{p-1}{p}}} \mathds{1}_{[\frac{1}{n},\frac{n-1}{n}]}$.  We choose $c, \hsp N, \hsp \varepsilon$ (in this order) such that \begin{itemize}
    \item There holds $\left( 1+ \frac{1}{c}\right)^{-\frac{2}{p}} > 1- \eta$.
    
    \item There holds $N \geq \ceil{\frac{c}{\delta}}$, so that \[ \int_{\frac{c}{n}}^{1-\frac{c}{n}}   \frac{1}{t^{\frac{1}{p}}\hsp (1-t)^{\frac{p-1}{p}}} \hsp\text{d}\mu(t) >  \int_{\delta}^{1-\delta}  \frac{1}{t^{\frac{1}{p}}\hsp (1-t)^{\frac{p-1}{p}}} \hsp\text{d}\mu(t) >(1-\eta)\hsp \mathcal{N}_{p}^{\mu}, \]for $n\geq N$.
    
    \item There holds $0< \varepsilon< \frac{1}{p}$. Moreover, denoting $w= \frac{1}{p}+\varepsilon$, we must choose $\varepsilon$ so that the following two conditions also hold: \be \int_{\delta}^{1-\delta} t^{-w}\hsp (1-t)^{w-1}\hsp \text{d}\mu(t) > (1-\eta)\hsp  \int_{\delta}^{1-\delta}  \frac{1}{t^{\frac{1}{p}}\hsp (1-t)^{\frac{p-1}{p}}} \hsp\text{d}\mu(t) \ee and such that \be \sum_{n=N}^{\infty} (n+1)^{-p\hsp w} > (1-\eta)\hsp \sum_{n=0}^{\infty} \hsp (n+1)^{-p \hsp w}. \ee The latter condition can be met for $\varepsilon$ sufficiently small, since the $p-$series diverges for $p=1$ and converges for $p>1$.
    \end{itemize}
Finally, we pick the sequence $a_n = (n+1)^{-w}$. Using Lemma \ref{lem15}, we have

\begin{equation}
    a_n = \frac{1}{\Gamma(\omega)}\hsp \int_0^{\infty} \mathrm{e}^{-(n+1)\hsp x}\hsp x^{w-1}\hsp \text{d}x
\end{equation}Using this and \eqref{ent}, we can get the following estimate for $e_n(t)$.

\begin{equation}
    \begin{split}
        e_n(t) &= \frac{1}{\Gamma(\omega)} \hsp \int_0^{\infty} \mathrm{e}^{-x}\hsp x^{w-1}\hsp \sum_{m=0}^{\infty} \binom{n+m}{m} \hsp t^m \hsp (1-t)^n \hsp \mathrm{e}^{-m\hsp x} \text{d}x \\ &= \frac{(1-t)^n}{\Gamma(\omega)}\hsp \int_0^{\infty} \mathrm{e}^{-x}\hsp x^{w-1}\hsp (1-t\hsp \mathrm{e}^{-x})^{-n-1}\hsp \text{d}x \\ &= \frac{1}{1-t}\hsp\frac{1}{\Gamma(\omega)} \hsp \int_0^{\infty} \mathrm{e}^{-x}\hsp x^{w-1}\hsp \left( \frac{1-t}{1-t\hsp \mathrm{e}^{-x}}\right)^{n+1}\hsp   \text{d}x \\ &\geq \frac{1}{1-t}\hsp\frac{1}{\Gamma(\omega)} \hsp \int_0^{\infty} \mathrm{e}^{-x}\hsp x^{w-1}\hsp \hsp \mathrm{e}^{-(n+1)\frac{t}{1-t}\hsp x}\hsp  \text{d}x \\&= \frac{1}{1-t} \left( \frac{1+n\hsp t}{1-t}\right)^{-w},
    \end{split}
\end{equation}where the penultimate step was obtained by Lemma \ref{lem14} and the last step from Lemma \ref{lem15}. When $\frac{c}{n}\leq t \leq 1-\frac{c}{n}$, we therefore have

\begin{equation}\label{214}
    \begin{split}
        e_n(t)&\geq \hsp (1-t)^{w-1}\hsp t^{-w}\hsp \left(1 +\frac{1}{c}\right)^{-w}\hsp (n+1)^{-w} \\ &\geq (1-t)^{w-1}\hsp t^{-w}\hsp (1-\eta)\hsp a_n.
    \end{split}
\end{equation}Using \eqref{214}, we have

\begin{equation}
    \begin{split}
        b_n = \int_0^1 e_n(t) \dmt &\geq \int_{\frac{c}{n}}^{1-\frac{c}{n}} e_n(t)\dmt \\&\geq (1-\eta)\hsp a_n\hsp\int_{\frac{c}{n}}^{1-\frac{c}{n}} t^{-w}\hsp (1-t)^{w-1}\dmt \\&\geq (1-\eta)\hsp a_n\hsp\int_{\delta}^{1-\delta} t^{-w}\hsp (1-t)^{w-1}\dmt \\ &\geq (1-\eta)^2\hsp a_n \hsp \int_{\delta}^{1-\delta} t^{-\frac{1}{p}}\hsp (1-t)^{\frac{1}{p}-1}\dmt \\&\geq (1-\eta)^3\hsp  a_n \hsp \mathcal{N}_{p}^{\mu}, \hspace{3mm}\text{for}\hspace{1mm}n\geq N,
    \end{split}
\end{equation}hence \begin{equation}
    \sum_{n=0}^{\infty}b_n^p \geq \sum_{n=N}^{\infty}b_n^p \geq \sum_{n=N}^{\infty} (1-\eta)^{3\hsp p}\hsp \left(\mathcal{N}_p^{\mu}\right)^p \hsp a_n^p \geq (1-\eta)^{3\hsp p+1}\hsp \left(\mathcal{N}_{p}^{\mu}\right)^p\hsp  \sum_{n=0}^{\infty}a_n^p.
\end{equation}
This establishes the "if" direction. For the "only if" direction, assume that $\mathcal{C}^{\mu}$ is bounded as an operator on $\ell^p$ for $1 \leq p < \infty$. We claim that the integral  $\mathcal{N}_p^{\mu}$ is finite. We shall make use of the same sequences as the ones just used. With the preceding notation, for an arbitrarily small but fixed $\eta >0$, we have for the sequences described above:

\begin{equation}
    b_n = \int_0^1 e_n(t) \hsp \text{d}\mu(t) \geq (1- \eta) \hsp a_n \hsp \int_0^1 (1-t)^{w-1} \hsp t^{-w} \hsp \text{d}\mu(t) 
\end{equation} and hence

\begin{equation}
    \left( \sum b_n^p\right)^{\frac{1}{p}} \geq \begin{bmatrix}(1-\eta) \hsp \int_0^1 (1-t)^{w-1} \hsp t^{-w} \hsp \text{d}\mu(t)\end{bmatrix} \hsp \left( \sum a_n^p\right)^{\frac{1}{p}}. \label{219}
\end{equation}From here we observe, given the boundedness of $\mathcal{C}^{\mu}$, that $\int_0^1 (1-t)^{w-1} \hsp t^{-w} \hsp \text{d}\mu(t)$ is finite and uniformly bounded above by the operator norm as $\eta \to 0$ . Moreover, it is not hard to establish that, as $\eta \to 0$, the way $\varepsilon$ is chosen means that $\varepsilon \to 0$ as well and hence the functions 

\[ (1-\eta)  (1-t)^{w-1} \hsp t^{-w} \to \frac{1}{(1-t)^{\frac{p-1}{p}}\hsp t^{\frac{1}{p}}} \]pointwise as $\eta \to 0$. Fix any $z > 0$ small. An application of the dominated convergence theorem on the interval $[z, 1-z]$ (on this fixed interval, the functions $ (1-\eta)  (1-t)^{w-1} \hsp t^{-w} $ are uniformly bounded above) now implies that the integral \[ \int_{z}^{1-z} \frac{1}{(1-t)^{\frac{p-1}{p}}\hsp t^{\frac{1}{p}}} \hsp \text{d}\mu(t) \]is finite and bounded above by the operator norm. As $z\to 0$, an application of the monotone convergence theorem gives the desired result.
\vspace{3mm}

\par\noindent We finally examine the case $p=\infty$. Assume that $\int_{0}^1 \frac{1}{1-t} \hsp \text{d}\mu(t)$ is finite and consider a sequence $(a_n) \in \ell^{\infty}$. We have 

\begin{equation}
\begin{split}
    b_n &= \sum_{m=0}^{\infty} \binom{n+m}{m} \int_0^1 (1-t)^n\hsp t^m \hsp a_m \hsp \text{d}\mu(t)\\ &\leq \begin{bmatrix}\int_0^1 \sum_{m=0}^{\infty} \binom{n+m}{m} (1-t)^n\hsp t^m \hsp \text{d}\mu(t)\end{bmatrix} \hsp \lVert (a_n) \rVert_{\ell^{\infty}}\\ &= \begin{bmatrix} \int_0^1 \frac{1}{1-t}\hsp \text{d}\mu(t)\end{bmatrix} \lVert (a_n) \rVert_{\ell^{\infty}}.
\end{split}
\end{equation}Taking the supremum over $n$, we have \be \lVert (b_n) \rVert_{\ell^{\infty}} \leq \int_{0}^1 \frac{1}{1-t}\hsp \text{d}\mu(t) \hsp \lVert (a_n) \rVert_{\ell^{\infty}}. \label{sup0} \ee By taking $(a_n)$ to be a constant sequence different than the zero sequence, we see that the constant is optimal.  Conversely, if the operator is bounded on $\ell^{\infty}$, picking any constant non-zero sequence $(a_n)$, we see that the integral $\int_0^1 \frac{1}{1-t}\hsp \text{d}\mu(t)$ is finite and the result follows.

\end{proof}

\subsection{The case where $\mu$ is allowed to have Dirac masses at $0$ and/or $1$}
To allow for the possibility of Dirac masses at $0$ or $1$, we need to be a bit more careful with semantics and notation. 
\subsubsection{Preliminary measure-theoretic notes}
Given a measure $\mu$ on $[0,1]$ and a function $f$ defined on $[0,1]$ which is measurable with respect to $\mu$, we shall denote the integral $\int_0^1 f(t) \hsp \text{d}\mu(t)$ by $\mu_{[0,1]} f$. In general, for any measurable subset $\mathcal{A}$ of $[0,1]$, we will use $\mu_{\mathcal{A}}(f)$ to denote the integral $ \mu_{[0,1]}(f\cdot \mathds{1}_{\mathcal{A}}) $, which is well-defined as $f\cdot \mathds{1}_{\mathcal{A}}$ is $\mu-$measurable. 

\vspace{3mm}

\par \noindent For a function $g$ defined only on a subset $\mathcal{P}$ of $[0,1]$, by a slight abuse of notation, we shall also use $\mu_{\mathcal{P}}(g)$ to denote the integral \[ \int_{\mathcal{P}} g \hsp \text{d}{\mu}_{\mathcal{P}}(t), \]where $\mu_{\mathcal{P}}$ is the restriction of the measure on $\mathcal{P}$ in the obvious way. 
\vspace{3mm}

\par \noindent For a given measure $\mu$ and for any function $f$ on $[0,1]$, the following decomposition holds:

\[ \mu_{[0,1]}(f) = \mu_{(0,1)}(f) + c_0 \hsp f(0) + c_1 \hsp f(1),\]where  $c_1 := \mu\left(\begin{Bmatrix} 0 \end{Bmatrix} \right)$ and $c_2:= \mu\left(\begin{Bmatrix}1 \end{Bmatrix} \right)$ are the Dirac masses of the measure at $0$ and $1$ respectively. This, in turn, provides the following decomposition for the matrix  $\mathcal{C}^{\mu}$:

\be  \mathcal{C}^{\mu}_{n,k} = \tilde{\mathcal{C}}_{n, k}^{\mu} + \hat{\mathcal{C}}_{n, k}^{\mu}, \label{deco} \ee where $\tilde{\mathcal{C}}_{n, k}^{\mu} := \mu_{(0,1)}\left( g_{n,k}(t) \right)$ and \[ \hat{\mathcal{C}}^{\mu}:= \begin{pmatrix} c_0 +c_1 & c_1 & c_1 & c_1 & \dots \\c_0 & 0 &0  & 0 & \dots\\ c_0 & 0&0 &0 & \ddots\\  c_0 & 0&0 &0 & \ddots\\ \vdots & \vdots& \vdots&\vdots &\ddots \\ 
\end{pmatrix}. \]
\subsubsection{Main Result}
In this section we will show the following:

\begin{theorem}
Assume that $\mu$ contains Dirac masses $c_0$ at $0$ and $c_1$ at $1$, with $\lvert c_0 \rvert + \lvert c_1 \rvert >0$. For $1\leq p <\infty$, the only case when $\mathcal{C}^{\mu}$ can be a bounded operator on $\ell^p$ is when $c_1\neq 0 = c_0$ and $p=1$. In that case, the operator is bounded if and only if

\[ \int_{(0,1]} \frac{1}{t}\hsp \text{d}\mu(t) =  \int_{(0,1)} \frac{1}{t}\hsp \text{d}\mu(t) + c_1 \]is finite. Moreover, this expression, when finite, equals the norm of the operator. Finally, when $p= \infty$, the operator is bounded if and only if $c_0\neq 0 = c_1$ and the integral \[ \int_{[0,1)} \frac{1}{1-t}\hsp \text{d}\mu(t)\]is finite. Moreover, the above, when finite, equals the norm of the operator on $\ell^{\infty}$.
\end{theorem}

\begin{proof}We begin by establishing the following two lemmata:

\begin{lemma}
If the measure $\mu$ contains a Dirac mass at $0$, then the matrix is unbounded on $\ell^p$ for $1\leq p < \infty$.
\end{lemma}

\par \noindent This holds for the following reason. The matrix $\hat{\mathcal{C}}^{\mu}$ for $c_0\neq 0$ does not even map $\ell^p$ to itself, since

\be \label{look}  \begin{pmatrix} c_0 +c_1 & c_1 & c_1 & c_1 & \dots \\c_0 & 0 &0  & 0 & \dots\\ c_0 & 0&0 &0 & \ddots\\  c_0 & 0&0 &0 & \ddots\\ \vdots & \vdots& \vdots&\vdots &\ddots \\ 
\end{pmatrix} \begin{pmatrix} a_0 \\a_1 \\ a_2 \\ a_3 \\  \vdots \end{pmatrix} = \begin{pmatrix} (c_0+c_1)\hsp a_0 + c_1 \hsp (a_1 + a_2+\dots) \\ c_0\hsp a_0 \\c_0\hsp a_0 \\c_0\hsp a_0\\ \vdots \end{pmatrix} \ee and clearly, as the  sequence on the right-hand side eventually consists of the same non-zero (as long as $a_0 \neq 0$) constant, it cannot be in $\ell^p$. Finally, since the matrix  $\tilde{\mathcal{C}}^{\mu}$ comprises only positive entries, the sum $\tilde{\mathcal{C}}^{\mu} + \hat{\mathcal{C}}^{\mu}$ also fails to map $\ell^p$ to $\ell^p$.
\vspace{3mm}
\begin{lemma}
If the measure $\mu$ contains a Dirac mass at $1$ and no Dirac mass at $0$, then the matrix $\mathcal{C}^{\mu}$ is unbounded as an operator on $\ell^p$ for $1 < p < \infty$.
\end{lemma}
\vspace{3mm}
\par\noindent To establish this, notice that the first term of the sequence on the right-hand side of \eqref{look} is $c_1 (a_0 + a_1 +\dots)$. In general, the $\ell^p$ spaces are nested, in the sense that $\ell^p \subset \ell^q$ for $p< q$ and hence for each $p>1$ there exists a sequence in $\ell^p$ which is not in $\ell^1$. Since $\tilde{\mathcal{C}}^{\mu}$ consists of only positive entries, it again follows that $\mathcal{C}^{\mu}$ does not map $\ell^p$ to itself when $c_1\neq 0 =c_0$ and $p>1$. The only case left to examine, when the measure contains a Dirac mass, is when $c_1 \neq 0 = c_0$ and $\tilde{\mathcal{C}}^{\mu}$ is viewed as an operator on $\ell^1$.
In this case we claim that the operator is bounded if and only if \[ \mu_{(0,1]} \left( \frac{1}{t}\right) = \int_{(0,1)} \frac{1}{t} \hsp \text{d}\mu(t) + c_1 \]is finite and that, when finite, the expression above equals the operator norm.

\vspace{3mm}
\par \noindent Indeed, assume first that the expression above is finite. Proceeding as in \eqref{use}, we get, for $0<t<1$:

\begin{equation}
    \sum_{n=0}^{\infty} e_n(t) \leq \frac{1}{t} \sum_{n=0}^{\infty} a_n(t)
\end{equation}for all $n$. We now notice that \be b_0 = \int_{(0,1)} e_0(t) \hsp \text{d}\mu(t) + c_1 \lVert (a_n) \rVert_{\ell^1} \label{eqn1} \ee and \be b_n = \int_{(0,1)} e_n(t) \hsp \text{d}\mu(t) \label{eqn2} \ee for $n \geq 1$. Adding \eqref{eqn1} and \eqref{eqn2} together, we arrive at \begin{equation}
    \lVert (b_n) \rVert_{\ell^1} \leq \left( \int_{(0,1)} \frac{1}{t} \hsp \text{d}\mu(t) + c_1 \right) \hsp \lVert (a_n) \rVert_{\ell^1}
\end{equation}To show that this is the optimal constant, we again use the sequences described in the previous section. With the same notation as before, with $a_n = (n+1)^{-w}$, we have

\begin{equation}
\begin{split}
    \sum_{n=0}^{\infty}b_n &= b_0 + \sum_{n=1}^{\infty} b_n = \int_{(0,1)} e_0(t) \hsp \text{d}\mu(t) + c_1 \lVert (a_n) \rVert_{\ell^1}+ \sum_{n=1}^{\infty} \int_{(0,1)} e_n(t) \hsp \text{d}\mu(t) \\ &\geq \left( (1-\eta)^{4}\int_{(0,1)}\frac{1}{t}\hsp \text{d}\mu(t) + c_1\right) \hsp\lVert (a_n) \rVert_{\ell^1}.
\end{split}
\end{equation}Letting $\eta \to 0$, we obtain the desired result.

\vspace{3mm}

\par \noindent Finally, let us suppose that the operator $\mathcal{C}^{\mu}$ is bounded on $\ell^1$. The same argument as in the previous case shows the finiteness of $\int_{(0,1]} \frac{1}{t} \hsp \text{d}\mu(t)$.

\vspace{3mm}
\par \noindent We finally examine the case $p=\infty$. The proof is similar in spirit to the previous results, although a bit simpler. Indeed, if $c_1\neq 0$ then it is clear from \eqref{deco} that $\mathcal{C}^{\mu}$ does not map $\ell^{\infty}$ to itself.

\vspace{3mm}
\par\noindent Assume that $\int_{[0,1)}\frac{1}{1-t} \hsp \text{d}\mu(t)$ is finite and consider a sequence $(a_n) \in \ell^{\infty}$. We have 

\begin{equation}\label{almost}
\begin{split}
    b_n &= \sum_{m=0}^{\infty} \binom{n+m}{m} \int_{(0,1)} (1-t)^n\hsp t^m \hsp a_m \hsp \text{d}\mu(t) +  c_0 \hsp a_n \\ &\leq \begin{bmatrix}\int_{(0,1)}\sum_{m=0}^{\infty} \binom{n+m}{m} (1-t)^n\hsp t^m \hsp \text{d}\mu(t)\end{bmatrix} \hsp \lVert (a_n) \rVert_{\ell^{\infty}} + c_0 a_n\\ &= \begin{bmatrix} \int_{(0,1)} \frac{1}{1-t}\hsp \text{d}\mu(t)\end{bmatrix} \lVert (a_n) \rVert_{\ell^{\infty}} + c_0 a_n.
\end{split}
\end{equation}Taking the supremum of \eqref{almost} over $n$, we have \be \lVert (b_n) \rVert_{\ell^{\infty}} \leq \int_{[0,1)} \frac{1}{1-t}\hsp \text{d}\mu(t) \hsp \lVert (a_n) \rVert_{\ell^{\infty}}. \label{sup} \ee By taking $(a_n)$ to be a constant sequence different than the zero sequence, we see that the constant is optimal.

\vspace{3mm}

\par \noindent Conversely, assuming that the operator is bounded from $\ell^{\infty}$ to itself, again picking a  non-zero constant sequence, we see that

\begin{equation}
    \begin{split}
        b_n &= \sum_{m=0}^{\infty} \binom{n+m}{m} \int_{(0,1)} (1-t)^n\hsp t^m \hsp a_m \hsp \text{d}\mu(t) +  c_0 \hsp a_n \\ &= \begin{bmatrix} \int_{[0,1)} \frac{1}{1-t}\hsp \text{d}\mu(t) \end{bmatrix} \hsp a_n.
    \end{split}
\end{equation}From this, the finiteness of the operator norm implies the finiteness of \[\int_{[0,1)} \frac{1}{1-t}\hsp \text{d}\mu(t)\] and the result follows.

\end{proof}
\noindent\textbf{Acknowledgements}
\vspace{3mm}

\noindent I would like to thank Aristomenis Siskakis for introducing me to this problem and for various beneficial suggestions regarding the format of the paper. I would also like to thank him, as well as Vasilis Daskalogiannis, Petros Galanopoulos and Giorgos Stylogiannis for several fruitful discussions on the topic. Finally, I thank Vasilis for useful comments and suggestions on the introduction.

\end{document}